\documentclass[11pt]{amsart}
\usepackage[mathlines, displaymath]{lineno}
\usepackage[margin=1in]{geometry}
\usepackage{amscd,amssymb, amsmath,  setspace, wasysym, yfonts, amsthm, mathtools, dirtytalk}
\usepackage{graphicx}
\usepackage[all]{xy}
\usepackage{tikz}
\usetikzlibrary{matrix}

\usepackage{url}
\usepackage{hyperref}

\theoremstyle{plain}
\newtheorem{theorem}{Theorem}
\newtheorem{prop}[theorem]{Proposition}

\newtheorem{claim}[theorem]{Claim}
\newtheorem{cor}[theorem]{Corollary}

\newtheorem*{conjhk}{Conjecture $(H_n^k)$}
\newtheorem{lemma}[theorem]{Lemma}

\theoremstyle{definition}

\newtheorem{defn}[theorem]{Definition}
\newtheorem{rmk}[theorem]{Remark}

\newtheorem*{ex*}{Example}

\newcommand\sO{{\mathcal O}}

\newcommand\sF{{\mathcal F}}

\newcommand\sE{{\mathcal E}}

\newcommand{\codim}{{\rm codim}\,}

\newcommand\rz{{\mathbf{Z}}}
\newcommand\rr{{\mathbf{R}}}
\newcommand\rc{{\mathbf{C}}}

\newcommand\rl{{\mathbf{L}}}

\newcommand\rd{{\mathbf{D}}}

\DeclareMathOperator{\Pic}{Pic}

\DeclareMathOperator{\Aut}{Aut}
\DeclareMathOperator{\Alb}{Alb}
\DeclareMathOperator{\Supp}{Supp}
\DeclarePairedDelimiterX\set[1]\lbrace\rbrace{#1}

\subjclass[2010]{14F05, 14E05} 
\keywords{Invariants of derived categories of sheaves, Rouquier isomorphism, fibrations, irregular varieties, non-vanishing loci}

\title[Derived equivalence and fibrations]
{Derived equivalence and fibrations over curves and surfaces}

\author{Luigi Lombardi}
\address{Department of Mathematics  \\ University of Milan\\Via Cesare Saldini 50, 20133 Milan, Italy }
 \email{\url{luigi.lombardi@unimi.it}}
 
\begin{document}

\begin{abstract}
We prove that the bounded  derived category  of  coherent sheaves on a smooth projective complex variety   reconstructs  the isomorphism classes of fibrations  onto smooth projective curves of genus $g\geq 2$. 
Moreover,  in dimension at most four, we   prove  that the same category reconstructs  the isomorphism classes of fibrations  onto  
normal projective surfaces with  positive holomorphic Euler characteristic and admitting a  finite morphism to an abelian variety. Finally, we study the derived invariance of a class of fibrations with minimal base-dimension under the condition that all the Hodge numbers of type  $h^{0,p}(X)$ are derived invariant.
\end{abstract}
\maketitle
\section{Introduction}

The bounded   derived category $\rd(X) \stackrel{{\rm def}}{=} D^b \big(\mathcal{C}oh(X)\big)$ of coherent sheaves   on a smooth projective complex variety $X$ is a homological object  that encodes several geometric properties of the variety itself.  For instance, if  the (anti)canonical bundle of $X$ is ample (\emph{resp.} big), then $\rd(X)$ reconstructs  $X$ up to isomorphism (\emph{resp.} $K$-equivalence). 
We recall that two smooth projective varieties $X$ and $Y$ are $K$\emph{-equivalent} if there exists a third smooth  projective variety $U$ and two birational morphisms $X \stackrel{\,\,p}{\leftarrow} U \stackrel{q}{\rightarrow} Y$ such that  $p^*\omega_X \simeq q^*\omega_Y$.
%
A leading conjecture concerning  the role of   $\rd(X)$ in birational geometry  is the $DK$-\emph{hypothesis} 
(\emph{cf}. \cite[Conjecture 4.4]{bondal+orlov:dercat}, \cite[Conjecture 1.2]{kawamata:birationalgeometry} and \cite[Conjecture 6.24]{huybrechts:fouriermukai}). 
This  predicts that $K$-equivalent varieties are $D$-equivalent, that is there exists an equivalence  of triangulated categories $\rd(X) \simeq \rd(Y)$.
%
In this paper, we prove that some classes of fibrations of projective varieties that  are invariant under $K$-equivalence,  are also invariant under $D$-equivalence.

To begin with,  we study the behavior of fibrations onto  smooth projective curves  under $D$-equivalence.
Let $X$ be a smooth projective complex variety and let  $g\geq 1$ be an integer.
An \emph{irrational pencil of genus $g$} of $X$ is a surjective morphism   with connected fibers $f\colon X \to C$ onto a smooth projective curve of genus $g$.
We say that two irrational pencils $f_i\colon X \to C_i$ $(i=1,2)$  are \emph{isomorphic} if there exists an isomorphism  $\gamma \colon C_1\xrightarrow{\sim} C_2$ such that $f_2 = \gamma \circ f_1$.  We set
\begin{equation}\label{sets}
F_X^g  \; = \; F^{1,g}_X \; \stackrel{{\rm def}}{=} \; \set[\big]{ \, \mbox{isomorphism classes of irrational pencils of genus } \, g \, \mbox{ of } X \,}.
\end{equation}
Our first result regards the derived invariance of irrational pencils onto smooth projective curves $C$
 of genus $g\geq 2$, or equivalently satisfying
 $\chi(\omega_C)>0$.
\begin{theorem}\label{mainthm}
Let $X$ and $Y$ be smooth projective complex varieties. If  $\Phi \colon \rd(X) \to \rd(Y)$ is an equivalence of triangulated categories,  then  for any integer $g\geq 2$ the equivalence   $\Phi$ induces  a bijection of sets $\mu_g\colon F_X^g \to F_Y^g$  preserving the bases of the fibrations. More specifically, if $\mu_g(f\colon X\rightarrow C)=(h\colon Y\rightarrow D)$, then the curves $C$ and $D$ are isomorphic.
\end{theorem}
  
The proof of Theorem \ref{mainthm} employs the results of \cite{lombardi+popa:nonvanishing} concerning the behavior of the \emph{non-vanishing loci}
\begin{equation}\label{introvi}
V^{i}(\omega_X)_0 \; \stackrel{{\rm def}}{=} \; \{ \, \alpha \in \Pic^0(X) \;  | \; H^{i}(X,\omega_X\otimes \alpha) \;   \neq \; 0 \, \}_0 , \quad \quad i=0,\ldots ,n\stackrel{{\rm def}}{=}\dim X
\end{equation}
attached to the canonical bundle 
under $D$-equivalence (the subscript $0$ denotes the  union of the irreducible components passing through the origin of an algebraic subset; \emph{cf}. \cite{green+lazar:deformations} and \cite{green+lazar:obstructions}).
Less informally, the definition of $\mu_g$   relies on the following two  ingredients:  the derived  invariance  of   $V^{n-1}(\omega_X)_0$, and the existence of a bijection between  $F_X^g$ and the set of  $g$-dimensional irreducible components of $V^{n-1}(\omega_X)_0$  (\emph{cf}.  Theorems \ref{lopo} and \ref{fibrAlbprim}). On the other hand, the isomorphism between the base curves of the fibrations is constructed by manipulating the support of  the kernel of the equivalence $\Phi$  in the style of Kawamata \cite{kawamata:dequivalence}. Earlier attempts at proving Theorem \ref{mainthm} appear in \cite[Remark 7.4]{lombardi:invariants} and \cite[Theorems 6(i) and 14]{lombardi+popa:nonvanishing}. Moreover, Theorem \ref{mainthm} answers affirmatively  a question posed in \cite[Question 13]{lombardi+popa:nonvanishing} (\emph{cf}. also \cite[Corollary 3.4]{popa:nonvanishingloci}), and carries applications towards the behavior of the properties of 
the fundamental group under derived equivalence (\emph{cf}. \S\ref{secapplications}). 

%

In the second part of this paper we study the derived invariance of a class of fibrations  
over normal projective surfaces,  which extends 
the case of irrational pencils of genus $g\geq 2$. 
 A $\chi$\emph{-positive $2$-higher irrational pencil}  is a 
 surjective morphism  with connected fibers $f\colon X\to S$   such that: $(i)$ $S$ is a normal projective surface, 
$(ii)$ $\chi(\omega_{\widetilde S})>0$  for some  resolution of singularities $\widetilde S \to S$, and $(iii)$ there exists a   morphism $S\to \Alb(\widetilde S)$ finite onto its image such that the composition $\widetilde S \to S \to \Alb(\widetilde S)$ equals the Albanese map of $\widetilde S$.
Two such $\chi$-positive $2$-higher irrational pencils are \emph{isomorphic} if there exists an isomorphism between the bases of the fibrations that  commutes with the structure morphisms (\emph{cf}. \S\ref{section:fibrations}). 
We consider the following sets of fibrations attached to $X$:
\begin{equation*}
F_X^{2,q} \stackrel{{\rm  def}}{=}  \{ \, \mbox{isomorphism classes of }\mbox{}\chi\mbox{-positive }\mbox{2-higher irrational pencils }{\scriptstyle f\colon X \rightarrow S}   \mbox{ such that  } {\scriptstyle h^0(\widetilde S,\Omega^1_{\widetilde S})=q}  \, \}.
\end{equation*}
\normalsize
It turns out that the behavior of $F_X^{2,q}$  under derived equivalence  is related to the conjectured derived invariance of the Hodge number $h^{0,2}(X)=\dim H^2(X,\sO_X)$.  In general, all  Hodge numbers are expected to be invariant under  $D$-equivalence. We refer to  \cite{huybrechts:fouriermukai}, \cite{popa+schnell:invariance} and  \cite{abuaf:units} for partial results in this direction (\emph{cf}. also Remark \ref{hodge} for a preview).

\begin{conjhk}\label{conjhodge}
Let $n\geq 1$ and $k\geq 0$ be integers such that 
$k \leq n$.  Then we have $h^{0,k}(Z_1)= h^{0,k}(Z_2)$ for all pairs of  smooth projective $D$-equivalent complex varieties $Z_1$ and $Z_2$    of dimension $n$.
\end{conjhk}
%

\begin{theorem}\label{mainthm2}
Suppose that   Conjecture $(H_n^2)$ holds for some   $n\geq 2$, and let $X$ and $Y$ be   smooth projective complex varieties of dimension $n$.
If $\Phi \colon \rd(X)\to \rd(Y)$ is an equivalence of triangulated categories, then for any   $q\geq 2$  the equivalence $\Phi$ induces a bijection of sets  $\nu_q\colon F_X^{2,q}\to F_Y^{2,q}$
 preserving the bases of the fibrations. 
\end{theorem}

A key tool  in the proof of Theorem \ref{mainthm2} is the derived invariance of  the non-vanishing locus $V^{n-2}(\omega_X)_0$. By the results of \cite{lombardi+popa:nonvanishing}, this invariance is guaranteed as soon as the Hodge number $h^{0,2}(X)$ is invariant.
Apart from this fact,  the proof of Theorem \ref{mainthm2} follows
 the general strategy of  Theorem \ref{mainthm}.  
It is important to note that the  result  \cite[Theorem 1.3]{abuaf:units} by Abuaf  proves Conjecture $(H^k_n)$ for all $0\leq k \leq n\leq 4$; therefore   Theorem \ref{mainthm2} holds unconditionally in dimension at most four. 

Finally, we extend the notion of $\chi$-positive $2$-higher irrational pencil to  fibrations with  base varieties of arbitrary dimension (\emph{cf}. \S\ref{section:fibrations}). Let $n=\dim X$ and define the non-negative integer
$$b_{\chi>0}(X) \; \in \; \{ 0, \ldots , n-1\}$$
as the minimal dimension of a base variety of a  $\chi$-positive $k$-higher irrational pencil of $X$ with $k\in \{1,\ldots , n-1\}$. We declare $b_{\chi>0}(X)=0$ if and only if $X$ does not carry any $\chi$-positive $k$-higher irrational pencil  for all  $k\in \{1,\ldots , n-1\}$.
The following result studies the derived invariance of $b_{\chi>0}(X)$. 

\begin{theorem}\label{introb}
Assume that Conjecture $(H_n^k)$ holds for some   $n\geq 2$ and  all integers $k\in \{1,\ldots ,n\}$. 
Let $X$ and $Y$ be smooth projective complex varieties of dimension $n$.  If $\Phi \colon \rd(X) \to \rd(Y)$ is an equivalence of triangulated categories,
then we have  $b_{\chi>0}(X) = b_{\chi>0}(Y)$.
Moreover, if $b\stackrel{{\rm def}}{=}b_{\chi>0}(X) >0$, then  for   any integer $q\geq b$ the equivalence $\Phi$ induces a bijection of sets  $\sigma_q\colon F_X^{b ,q}\to F_Y^{ b ,q}$
 preserving the bases of the fibrations. 
\end{theorem} 

As in a higher dimension the correspondence between  $\chi$-positive higher irrational pencils and irreducible components of non-vanishing loci is not in general one-to-one, the techniques of this paper are not enough to establish the $D$-invariance of all the sets  $F_X^{k,q}$  even if Conjecture $(H_n^k)$ would hold for all $k\geq 0$.
Along these lines, Caucci and Pareschi have interesting results concerning the derived invariance of fibrations of varieties of maximal Albanese dimension (\emph{cf}.  \cite[\S 4]{caucci+pareschi:invariants}).
Finally we refer to \S\ref{secfurther} for partial results  about the    invariance of fibrations over threefolds of $D$-equivalent fourfolds.
 

\section{Background material}\label{secpreview}
 
 \subsection{Notation} Throughout the paper we work over the field of complex numbers.
 By the term \emph{fibration} we mean a morphism of varieties that is surjective and with connected fibers. We denote by $q(X)=h^1(X,\sO_X) =  \dim \Pic^0(X)$ the \emph{irregularity} of a smooth projective variety $X$, and denote  by $alb_X\colon X\to \Alb(X)$ the Albanese map. 
  The morphism $alb_X$ is only defined  up to the choice of a point in $X$ and any two Albanese maps  differ by a translation on $\Alb(X)$.
We say that $X$ is of \emph{maximal Albanese dimension} if $alb_X$ is generically finite onto its image. 
The irregularity of a normal variety  $V$  is defined as $q(V)\stackrel{{\rm def}}{=}q(\widetilde V)$ where $\widetilde V$ is any resolution of singularities.  
We say that  a fibration  $\widetilde f\colon \widetilde X\to \widetilde V$ is a \emph{non-singular representative} of a fibration $f\colon X\to V$ if there exist resolutions of singularities 
 $\pi \colon \widetilde X  \to X$ and $\rho\colon \widetilde V \to V$  such that $f\circ \pi = \rho \circ \widetilde f$ (\emph{cf}. \cite[(1.10)]{mori:classification}).
 
 \subsection{The Rouquier isomorphism}\label{secrouquier}
 Let $X$ and $Y$ be smooth projective varieties of dimension $n$. 
We denote by $\rd(X) = D^b \big(\mathcal{C}oh(X) \big)$ and $\rd(Y) = D^b \big(\mathcal{C}oh(Y) \big)$ the bounded derived categories of coherent sheaves on $X$ and $Y$, respectively. We say that  $X$ and $Y$ are $D$\emph{-equivalent} if there exists an equivalence of triangulated categories $\Phi \colon \rd(X) \stackrel{\sim}{\longrightarrow} \rd(Y)$.
By \cite[Theorem 2.2]{orlov:representability}, this equivalence is of  Fourier--Mukai type $\Phi(-) \simeq \Phi_{\sE}(-) \stackrel{{\rm def}}{=}\mathbf{R} p_{2*}\big(p_1^*(-)\stackrel{\mathbf{L}}{\otimes } \sE \big)$. Here  the morphisms
$p_1$ and $p_2$ are the  projections from $X\times Y$ onto the first and second factor, respectively, and 
$\sE$ is an object in  $\mathbf{D}(X\times Y) \stackrel{\rm def}{=}D^b\big(\mathcal{C}oh(X\times Y) \big)$,
 uniquely determined by $\Phi$ up to isomorphism.  
Any  equivalence $\Phi_{\sE}$  induces an isomorphism of algebraic groups
 $$F_{\sE} \; \colon  \; \Aut^0(X) \, \times \, \Pic^0(X) \; \stackrel{\sim}{\longrightarrow} \; \Aut^0(Y) \, \times \, \Pic^0(Y)$$ 
 called  \emph{Rouquier's isomorphism} (\emph{cf}. \cite[Th\'{e}or\`{e}me 4.18]{rouquier:automorphisms} and \cite[Equation (3.1)]{popa+schnell:invariance}). We refer to equation \eqref{rouquieraction} for the   action of $F_{\sE}$. 
%
%
%

Let 
$$ V^k(\omega_X) \; \stackrel{{\rm def}}{=} \; \{\, \alpha \in \Pic^0(X) \,| \, H^k(X , \omega_X \otimes \alpha) \neq 0 \,\}$$
be the $k$-th non-vanishing locus attached to the canonical bundle.
The following theorem describes the images of the  loci $V^{k}(\omega_X)_0$ under $F_{\sE}$ (\emph{cf}. \eqref{introvi} for the definition of $V^{k}(\omega_X)_0$).
  We refer to \cite[Theorem 12]{lombardi+popa:nonvanishing} for a more general result.
 \begin{theorem}[Lombardi--Popa]\label{lopo}
 Suppose  that Conjecture $(H_n^k)$ holds for some integers $n\geq 1$ and $k \in \{0, \ldots , n\}$. If $X$ and $Y$ are smooth projective   varieties of dimension $n$ and  $\Phi_{\sE} \colon \rd(X)\stackrel{\simeq}{\longrightarrow} \rd(Y)$ is an equivalence, then we have
\begin{equation}\label{rouqv} 
F_{\sE} \big({\rm id}_X,V^{n-k}(\omega_X)_0 \big) \; = \; \big({\rm id}_Y, V^{n-k}(\omega_Y)_0 \big). 
\end{equation}
  In particular,  $F_{\sE}$ induces an isomorphism of    algebraic sets 
 $V^{n-k}(\omega_X)_0 \simeq V^{n-k}(\omega_Y)_0$.
\end{theorem}

\begin{rmk}\label{hodge}
Conjecture $(H_n^k)$  holds for all integers $n\geq 1$ and $k \in \{0,1,n-1,n\}$ (\emph{cf}. \cite[Remark 5.40]{huybrechts:fouriermukai}, \cite{popa+schnell:invariance} and  \cite{lombardi+popa:nonvanishing}). Moreover, it holds for all $n\leq 4$ and  $k\in \{0, \ldots , n \}$ (\emph{cf}. \cite[Theorem 1.3]{abuaf:units}).
\end{rmk}
 
The above Theorem \ref{lopo}  plays an important role in the  construction of the bijections $\mu_g$ of Theorem \ref{mainthm}. This goes as follows.  By  Green--Lazarsfeld's  Theorem \cite[Theorem 0.1]{green+lazar:obstructions}, the positive-dimensional irreducible components of $V^{n-1}(\omega_X)_0 $ are abelian varieties that give rise to irrational pencils $f\colon X\to C$ of genus $g\geq 1$ (see Lemma \ref{5.1} and Remark \ref{rmkalb}). Moreover, 
in \cite[Corollaire 2.3]{beauville:h1},  Beauville proves that the irrational pencils $f$ must be  of genus  $g\geq 2$,  and hence that 
$$V^{n-1}(\omega_X)_0 \; = \; \{\sO_X\}  \, \cup  \, \bigcup_{g\geq 2}\;\; \bigcup_{(f\colon X \to C) \in  F^g_X}  \big(f^*\Pic^0(C)\big)$$ if $q(X)>0$
(see \eqref{sets} for the definition of $F_X^g$). In  Theorem \ref{fibrAlbprim} 
we will observe that for any integer $g\geq 2$
 there  are one-to-one correspondences  of sets 
$$u_{X,1,g}\,\colon\, F_X^{g}\, \to \, \pi_0^g\big( V^{n-1}(\omega_X)_0 \big), \quad \big(f\colon X\to C \big)\, \mapsto \, f^*\Pic^0(C)$$ 
between $F_X^g$  and the set  of $g$-dimensional irreducible components of $V^{n-1}(\omega_X)_0$.
In view of Theorem \ref{lopo} and Remark \ref{hodge}, we define the bijections $\mu_g\colon F_X^g\to F_Y^g$  as  $\mu_g\,  = \,  u_{Y,1,g}^{-1} \, \circ \, F_{\sE} \, \circ \, u_{X,1,g}$.

\section{Non-vanishing loci and fibrations}\label{section:fibrations}
We denote by  $X$   a smooth projective   variety of dimension $n$. 
\subsection{Non-vanishing loci}
The \emph{non-vanishing loci} attached to a coherent sheaf $\sF$ on $X$ are the  algebraic closed subsets of $\Pic^0(X)$ defined by
$$V^i(\sF) \; \stackrel{{\rm def}}{=} \; \{ \, \alpha \, \in \, \Pic^0(X) \; | \; H^i(X,\, \sF\otimes \alpha) \, \neq  \, 0 \, \} \quad (i\geq 0). $$ 
We denote by $V^i(\sF)_0$ for the union of all the components of $V^i(\sF)$ passing through the origin.

Of particular interest is the case of the canonical bundle $\sF = \omega_X$.
Fundamental theorems of Green--Lazarsfeld and Simpson (\emph{cf}. \cite[Theorem 0.1]{green+lazar:obstructions} and \cite{simpson:subspaces}) prove that 
every irreducible component of $V^i(\omega_X)$  is a translate of an abelian subvariety of $\Pic^0(X)$
by a point of finite order. Furthermore,  in \cite[Theorem 1]{green+lazar:deformations}, the authors prove 
that 
if $X$ is of maximal Albanese dimension, then   the \emph{generic vanishing condition}  
\begin{equation}\label{gvcondition}
\codim V^i(\omega_X) \geq i \quad \mbox{ for all } \; i>0
\end{equation}
 holds.

\begin{lemma}\label{lemchi}
Suppose that $X$ is of maximal Albanese dimension. Then we have  $\chi(\omega_X)  \geq 0$. Moreover, we have 
  $\chi(\omega_X)>0$ if and only if $V^0(\omega_X) = \Pic^0(X)$.
\end{lemma}  
\begin{proof}
By the inequalities in  \eqref{gvcondition} we have that $V^i(\omega_X) \neq \Pic^0(X)$ for all $i>0$.
Hence, by the deformation invariance of the holomorphic   Euler characteristic, one deduces that 
$$\chi(\omega_X) \, = \,  \chi(\omega_X \otimes \alpha) \, = \, h^0(X,\omega_X \otimes \alpha) \, \geq \, 0 \quad \mbox{ for any } \quad \alpha \in \Pic^0(X) \backslash \Big( \bigcup_{i>0} V^i(\omega_X)\Big).$$
\end{proof}
%

%
\begin{prop}\label{picbir}
Let $U$ be a normal projective variety, and let $\vartheta \colon \widetilde U \to U$ be a resolution of singularities. If there exists a morphism $a\colon U\to \Alb(\widetilde U)$ such that 
the composition $\widetilde U \stackrel{\vartheta }{\to} U \stackrel{a}{\to} \Alb(\widetilde U)$ equals  the Albanese map $alb_{\widetilde U}$, then the morphisms  $\vartheta ^*$ and $\vartheta _*$ induce an isomorphism of algebraic groups $\Pic^0(\widetilde U) \simeq \Pic^0(U)$. 
\end{prop}

\begin{proof}
  The morphism $\vartheta ^*\colon \Pic^0(U)\to \Pic^0(\widetilde U)$ is surjective as any line bundle $\alpha \in \Pic^0(\widetilde U)$ is a  pullback of a topologically trivial line bundle on $\Alb(\widetilde U)$. Suppose now that 
 $\vartheta ^*(\alpha_1\otimes \alpha_2^{-1}) \simeq \sO_{\widetilde U}\simeq \vartheta ^*\sO_{U}$ where $\alpha_1$ and $\alpha_2$ are   in $\Pic^0(U)$. As  $\vartheta _*\sO_{\widetilde U}\simeq \sO_U$, the projection formula yields $\alpha_1\otimes \alpha_2^{-1}\simeq \sO_U$.
  \end{proof}

\begin{rmk}\label{invvi}
If $\vartheta \colon \widetilde U\to U$ is a birational morphism between    smooth projective varieties, then  $\vartheta^* \colon \Pic^0(U) \to \Pic^0(\widetilde U)$ induces isomorphisms
 $V^i(\omega_{U}) \simeq  V^i(\omega_{\widetilde U})$ for all $i\geq 0$.
\end{rmk}

The following proposition will be useful in the sequel.
\begin{prop}\label{extension}
Let $U$ be a proper  variety and  let $a\colon U \to A$ be a    morphism to an abelian variety, finite onto its image. 
Then any   rational map $Y \dashrightarrow U$ from a smooth projective variety $Y$ extends to a morphism.
\end{prop}
\begin{proof}
 If $Y\dashrightarrow  U$ is  not a morphism, then by \cite[Corollary 1.5]{KM} the variety $U$   contains a rational curve $E$. Therefore $a(E)$ is  a rational curve in $A$, which is a contradiction.
\end{proof}

 \subsection{Fibrations}
In this subsection we establish the existence of the bijection $u_{X,1,g}$ mentioned in \S\ref{secpreview} together with its extension to  fibrations over higher-dimensional bases.  We begin by defining the following class of varieties (\emph{resp}. fibrations) which represents a  higher-dimensional analog of that of smooth projective  curves  (\emph{resp.}   irrational pencils) of genus  $g\geq 2$.

\begin{defn}\label{defchipos}
A normal projective variety $V$ is   $\chi$\emph{-positive} if there exists a resolution of singularities $\rho \colon \widetilde V \to V$  and a  morphism
$b\colon V \to \Alb(\widetilde V)$ finite onto its image such that  the following two conditions hold: $(i)$ the composition 
$\widetilde V \stackrel{\rho}{\longrightarrow} V\stackrel{b}{\longrightarrow} \Alb(\widetilde V)$ equals the Albanese map $alb_{\widetilde V}$, and $(ii)$  $\chi(\omega_{\widetilde V})>0$.
\end{defn}
\begin{defn}
Let $0< k< n$ be an integer.
A  $\chi$-\emph{positive} $k$-\emph{higher irrational pencil} of $X$ is a fibration $f\colon X \to V$ onto a $\chi$-positive variety of dimension $k$.
Two such fibrations  $f_1 \colon X \to V_1$  and $f_2 \colon X \to V_2$ are \emph{isomorphic} if there exists an isomorphism $\varphi\colon V_1\stackrel{ \sim}{\to} V_2$ 
such that $f_2=\varphi \circ f_1$. 
\end{defn} 

\begin{rmk}\label{rmkbirational}
Let $\widetilde V_i \stackrel{\rho_i}{\longrightarrow} V_i\stackrel{b_i}{\longrightarrow} \Alb(\widetilde V_i)$ ($i=1,2$) be varieties and morphisms as in Definition \ref{defchipos}.
Any birational map $\psi \colon V_1 \dashrightarrow V_2$ induces
a birational map $\widetilde V_1 \dashrightarrow \widetilde V_2$, and thus
 an isomorphism $\psi_*\colon \Alb(\widetilde V_1) \to \Alb(\widetilde V_2)$ such that $b_2 \circ \psi = \psi_* \circ b_1$
 (\cite[Lemma 2.6]{Ueno}). As $b_1$ and $b_2$ are finite morphisms, the map $\psi$ extends to an isomorphism. We conclude that in the definition of an  isomorphism of  $\chi$-positive higher irrational pencils we may just require  $\varphi$ to be a birational map. 
\end{rmk}
 
\begin{rmk}\label{rmkresolution}
Let $\widetilde V \stackrel{\rho}{\to} V\stackrel{b}{\to} \Alb(\widetilde V)$ be varieties and morphisms as in Definition \ref{defchipos}.
If $V' \to V$ is   another resolution of singularities, then $b$ induces a morphism $b'\colon V \to \Alb(V')$ finite onto its image such that  the composition $V' \to V \stackrel{b'}{\to} \Alb(V')$ equals the Albanese map of $V'$. This fact is again an application of \cite[Lemma 2.6]{Ueno}. 
Moreover, once  a resolution $\widetilde V$ of $V$ has been fixed, a finite map $b\colon V \to \Alb(\widetilde V)$ as in Definition \ref{defchipos} is uniquely determined up to 
a translation on $\Alb (\widetilde V)$.
\end{rmk}

 \begin{rmk}
 A $\chi$-positive variety $V$ is necessarily of general type (\emph{cf}. \cite[Theorem 1]{chen+hacon:pluricanonicalmaps}). Moreover, if $V$ is non-singular, then $V$ does not admit any non-trivial Fourier--Mukai partner.
 \end{rmk}
 
 \begin{rmk}
 In \cite{catanese:moduli}, the author considers fibrations  over normal projective varieties of maximal Albanese dimension and with non-surjective Albanese map. We notice that our classes of  $\chi$-positive higher irrational pencils are not  sub-classes of Catanese's fibrations. In fact, there are smooth projective varieties with  positive holomorphic Euler characteristic   such that their Albanese maps are finite and surjective (\emph{cf}. \cite[Example 7.2]{barja+pardini+stoppino:linearsystems}).
 \end{rmk}
 For any pairs  of integers  $q\geq k\geq 1$ we define the following sets of fibrations:
$$\small F_X^{k,q} \, \stackrel{{\rm def}}{=}  \, \{\,  \mbox{isomorphism classes of }\chi\mbox{-positive }k\mbox{-higher irrational pencils }f\colon X \rightarrow V
  \mbox{ such that  } q(V)=q  \,\}.$$
Note that, for any integer $g\geq 2$, the   set $F_X^{1,g}$ coincides with the set $F_X^g$ of the Introduction.
The main result of this section is Theorem \ref{fibrAlbprim} which builds upon the results of \cite{green+lazar:obstructions} and \cite{pareschi:standard}. 
To begin with, we recall the result \cite[Lemma 5.1]{pareschi:standard}.
 
\begin{lemma}[Green--Lazarsfeld, Pareschi]\label{5.1}
Let $0<i<n$ be an integer and let 
 $Z\subset V^i(\omega_X)_0$ be an irreducible component of positive dimension. Then $Z$ induces a fibration $p\colon X\to V$ onto a normal projective variety $V$ of dimension $0<\dim V\leq n-i$ (see Remark \ref{rmkalb} for the definition of $p$). Moreover,  if $(\pi \colon \widetilde X \to X, \; \rho \colon \widetilde V \to V, \; \widetilde p\colon \widetilde X\to \widetilde V)$  is a non-singular representative of $p$, then
  $V$ admits a   morphism 	 $b\colon V \to \Alb(\widetilde V)$ finite onto its image such that $b \circ \rho = alb_{\widetilde V}$.  Moreover, we have 
  $$\chi(R^i\widetilde p_*\omega_{\widetilde X})>0, \quad V^0(R^i\widetilde p_*\omega_{\widetilde X}) = \Pic^0(\widetilde V), \quad \mbox{ and }  \quad 
 Z \, = \, \pi_* \widetilde p^*V^0(R^i\widetilde p_*\omega_{\widetilde X}) \, = \, \pi_* \widetilde p^*\Pic^0(\widetilde V).$$
\end{lemma}

\begin{rmk}\label{rmkalb}
The fibration $p\colon X\to V$ of Lemma \ref{5.1} is defined as  the Stein factorization of the composition $(q \, \circ \, alb_X)\colon X \to \widehat Z$ where
 $q \colon\Alb(X)\to \widehat Z$ is the dual map of the inclusion $Z\subset \Pic^0(X)$. In particular, $V$ admits a  morphism $b\colon V \to \widehat{Z}$  which is finite onto its image. In addition, if $\widetilde V \to V$ is any resolution of singularities, then  $\widehat{Z} \simeq \Alb(\widetilde V)$ and   the composition $\widetilde V \to V \stackrel{b}{\to} \Alb(\widetilde V)$ equals $alb_{\widetilde V}$ (\emph{cf}. \cite[Lemma 5.1]{pareschi:standard}). It follows that  $q(V) = \dim Z$. 
\end{rmk}

\begin{rmk}
As the Albanese map is defined up to  the choice of a point, the fibration $p\colon X \to V$ of Lemma \ref{5.1} is only determined  up to 
isomorphism of fibrations.  
\end{rmk}

\begin{cor}\label{corkollar}
Assume that  the hypotheses of Lemma \ref{5.1} hold. If $\dim V = n-i$, then we have  $\chi(\omega_{\widetilde V}) >0$ and the fibration
 $f\colon X \to V$ determines  a class  in $F_X^{n-i,q(V)}$.
  \end{cor}

\begin{proof}
By \cite[Proposition 7.6]{kollar:higher1} there is an isomorphism  $R^i \widetilde f_*\omega_{\widetilde X}\simeq \omega_{\widetilde V}$. It follows that $\chi( \omega_{\widetilde V}) = \chi(R^i \widetilde f_*\omega_{\widetilde X}) >0$. 
\end{proof}

\begin{lemma}\label{gamma2}
Let $f_i\colon X \to V_i$ $(i=1,2)$ be two fibrations onto normal projective varieties such that $\dim V_2\leq \dim V_1$. Let $(\pi_i\colon \widetilde X_i \to X, \, \rho_i \colon \widetilde V_i \to V, \, \widetilde f_i \colon \widetilde X_i \to \widetilde V_i )$ be non-singular representatives of $f_i$ for $i=1,2$, and let $b_i\colon V_i \to \Alb( \widetilde V_i)$ be morphisms finite onto their images such that $b_i \circ \rho_i = alb_{\widetilde V_i}$. If $\pi_{1*} \widetilde f^*_1 \Pic^0(\widetilde V_1) \subset  \pi_{2*} \widetilde f^*_2 \Pic^0(\widetilde V_2)$ in $\Pic^0(X)$, then there exists an isomorphism $\gamma \colon V_2 \stackrel{\sim}{\to} V_1$ such that $f_1 = \gamma \circ f_2$.
\end{lemma}

\begin{proof}
We consider the following commutative diagram
$$
\centerline{ \xymatrix@=32pt{
 X \ar[r]^{f_2}\ar@/^2pc/[rr]^{f_1} \ar[d]^{alb_X} & V_2\ar[d]^{b_2}\ar@{.>}[r]^{\exists\, \,\gamma} & V_1\ar[d]^{b_1} \\
 {\rm Alb}(X) \ar[r]^{\widetilde f_{2*}} \ar@/_2pc/[rr]^{\widetilde f_{1*}} & {\rm Alb}(\widetilde V_2)\ar[r]^{i_*} & {\rm Alb}(\widetilde V_1), \\}} 
 \noindent
  $$   
where the morphisms $\widetilde f_{i*}$ ($i=1,2$) are induced by the inclusions $\pi_{i*} \widetilde f^*_i \Pic^0(\widetilde V_i) \subset \Pic^0(X)$, or, equivalently, by  the morphisms  $\widetilde{f_i}$.
 Moreover, the morphism $i_*$ is induced by the inclusion $\pi_{1*}\widetilde f_1^*\Pic^0(\widetilde V_1)\subset \pi_{2*} \widetilde f_2^*\Pic^0(\widetilde V_2)$. Also note   both the inner and outer squares of the above  diagram commute.
As the general fiber of $f_2$ is contracted by $f_1$ (recall that $b_1$ is finite onto its image), by \cite[Lemma 1.15]{debarre:higher} there exists a dominant rational map $\gamma\colon V_2\dashrightarrow V_1$ such that $f_1=\gamma\circ f_2$. 
Hence $q(\widetilde V_1) = q(\widetilde V_2)$ and $i_*$ is an isomorphism.
Moreover  $\gamma$ is  birational as $\dim V_1=\dim V_2$ and  $f_1, f_2$ have connected fibers. 
Finally,   $b_1\circ \gamma = i_* \circ  b_2$, so that  $\gamma$ extends to an  isomorphism as $b_1, b_2$ have finite fibers.
\end{proof} 

\begin{cor}\label{corinverse}
  Let $0<i<n$ be an integer and  let $Z \subset V^i(\omega_X)_0$ be an  irreducible component of positive dimension inducing 
  a fibration $p\colon X \to V$ as in Lemma \ref{5.1}. Let $f\colon X \to Q$ be a fibration onto a normal  projective variety  such 
   that $Z= \pi_* \widetilde f^*\Pic^0(\widetilde Q)$ for some non-singular representative $(\pi \colon \widetilde X \to X, \rho \colon \widetilde Q \to Q, \widetilde f \colon \widetilde X \to \widetilde Q)$ of $f$. Assume  there exists a  morphism   $b\colon Q \to \Alb(\widetilde Q)$ finite onto its image such that $b\circ \rho = alb_{\widetilde Q}$.
      Then there exists an isomorphism $\gamma \colon Q \stackrel{\sim}{\to}  V$ such that $p = \gamma \circ f$.
   \end{cor}
   \begin{proof}
   By Remark \ref{rmkalb} the Stein factorization of the composition $X \to \Alb(X) \to \widehat Z$ factors through a finite morphism $b'\colon V \to \widehat Z$.
We apply Lemma \ref{gamma2}.
   \end{proof}
   
\begin{cor}\label{corfibration}
Let $f\colon X \to V$ be a fibration onto a normal projective variety of dimension $k$, and let $(\pi \colon \widetilde X \to X, \, \rho \colon \widetilde V \to V, \, \widetilde f \colon \widetilde X \to \widetilde V)$ be a non-singular representative of $f$. Assume   there exists a  morphism  $b\colon V \to \Alb(\widetilde V)$ finite onto its image such that $b\circ \rho = alb_{\widetilde V}$. If $\chi(\omega_{\widetilde V} ) > 0$, then $\pi_*\widetilde f^*\Pic^0(\widetilde V)$  is an irreducible component  of $V^{n-k}(\omega_X)_0$ of dimension $q(V)$. Moreover  $\pi_*\widetilde f^*\Pic^0(\widetilde V)$  does not depend on the particular  choice of the non-singular representative $(\pi, \rho, \widetilde f)$. Finally, the fibration induced by the component $\pi_*\widetilde f^*\Pic^0(\widetilde V)$   by  Lemma \ref{5.1} is precisely the fibration  $f\colon X \to V$ up to isomorphism. 
\end{cor}   
   
   \begin{proof}
By  Lemma \ref{lemchi} we have $V^0(\omega_{\widetilde V})=\Pic^0(\widetilde V)$. Moreover, by \cite[Lemma 6.3]{lombardi:invariants} and Remark \ref{invvi}, the variety $Z\stackrel{{\rm def}}{=} \pi_* \widetilde f^*\Pic^0(\widetilde V)$ is contained  in $V^{n-k}(\omega_ X)_0$. We now show    $Z$ is an irreducible component of $V^{n-k}(\omega_X)_0$. We argue by contradiction and assume   $Z$ is strictly contained in an irreducible component $Z'$ of $V^{n-k}(\omega_{X})_0$. By Lemma \ref{5.1},  $Z'$ gives rise to both a fibration $X\to U$ such that $0<\dim U\leq k$ as well as  a morphism $U \to \Alb(\widetilde U)$ finite onto its image satisfying the properties listed  in Lemma \ref{5.1} (here $\widetilde U$ is a resolution of singularities of $U$). Because of Lemma \ref{gamma2} this is impossible. The other assertions of the corollary  follow  by    Lemma \ref{gamma2}.
%
   \end{proof}

In the following results we denote by $(\pi \colon \widetilde X \to X, \, \widetilde f \colon \widetilde{X}\to \widetilde{V})$ 
a non-singular representative of a $\chi$-positive higher irrational pencil $f\colon X\to V$, and by  $\pi_0^q\big(V^{n-k}(\omega_X)_0\big)$ the set of     
$q$-dimensional irreducible components of $V^{n-k}(\omega_X)_0$.

\begin{prop}\label{prop:bijinj}
Let $q\geq k\geq 1$ be integers. 
Then the assignment
$$u_{X,k,q}\colon F_X^{k,q} \, \to \, \pi_0^q\big(V^{n-k}(\omega_X)_0\big),\quad \big(f\colon X\to V \big) \, \mapsto \, \pi_* \widetilde{f}^*\Pic^0(\widetilde V)$$ 
is well-defined and injective. In particular the sets $F_{X}^{k,q}$ are finite.
\end{prop}

\begin{proof}
Let $(f\colon X\rightarrow V) \in F_X^{k,q}$ be an isomorphism class determined by a  $\chi$-positive $k$-higher irrational pencil $f\colon X \to V$ with $q(V)=q$.   By Corollary  \ref{corfibration}, the abelian variety $\pi_* \widetilde f^*\Pic^0(\widetilde V)$ is an irreducible component of  $V^{n-k}(\omega_X)_0$ of dimension $q$. This variety is independent both 
of the non-singular representative $(\pi, \widetilde f)$, and 
the choice of a representative of the class of $f$. 
Hence  the map $u_{X,k,q}$ is injective by Corollary \ref{corinverse}.
\end{proof}

\begin{prop}\label{prop:bijsurj}

Let $q\geq k\geq 1$ be integers and suppose   $\dim V^i(\omega_X)_0\leq 0$ for all $i>n-k$.
Then the assignment
$$u_{X,k,q}\colon F_X^{k,q} \, \to \, \pi_0^q\big(V^{n-k}(\omega_X)_0\big),\quad \big(f\colon X\to V \big) \, \mapsto \, \pi_* \widetilde{f}^*\Pic^0(\widetilde V)$$
of Proposition \ref{prop:bijinj} is surjective.
\end{prop}

\begin{proof}
Let  $Z\subset V^{n-k}(\omega_X)_0$ be an irreducible component of dimension $q$. We will show that $Z$ arises from an isomorphism class of a  $\chi$-positive $k$-higher irrational pencil of irregularity $q$.
By Lemma \ref{5.1}, the component   $Z$ determines a $\chi$-positive $k$-higher irrational pencil $f \colon X\to V$ up to isomorphism onto a normal projective variety $V$ such that $0<\dim V\leq k$ and $q(V)=q$. Moreover, 
we have  $Z=\pi_*\widetilde f^* \Pic^0( \widetilde V)$; also, 
 $V$ admits a finite  morphism, onto its image, to $\Alb(\widetilde V)$.
Hence,  by Corollary \ref{corkollar}, it is enough to show    $\dim V=k$.
We proceed  by contradiction and suppose   that $\dim V=k-s$ for some   integer  $s>0$. We distinguish two cases:  $\dim V^0(\omega_{\widetilde V})_0>0$ and $\dim V^0(\omega_{\widetilde V})_0=0$ (notice that in either case $\sO_{\widetilde V} \in V^0(\omega_{\widetilde V})_0$). 
If $\dim V^0(\omega_{\widetilde V})_0>0$, then by pulling topologically trivial line bundles back to $\widetilde X$,
 we   have   $ V^0(\omega_{\widetilde V})_0 \subset V^{n-k+s}(\omega_{\widetilde X})_0 \simeq V^{n-k+s}(\omega_X)_0$ (\emph{cf}. \cite[Lemma 6.3]{lombardi:invariants}) which is a contradiction. Suppose now the second case $V^0(\omega_{\widetilde V})_0=\{\sO_{\widetilde V}\}$. 
Then, by \cite[Proposition 2.2]{ein+laz:theta}, the Albanese map   $alb_{\widetilde V}\colon \widetilde V\to \Alb(\widetilde V)$ is  surjective and we   have  $q= q(V)=q(\widetilde V)\leq \dim \widetilde V=k-s<k$. This is  again a contradiction. 

\end{proof}

By putting together the previous two propositions we  obtain the following result.

\begin{theorem}\label{fibrAlbprim}
Let $q\geq k\geq 1$ be integers and suppose   $\dim V^i(\omega_X)_0\leq 0$ for all $i>n-k$.
Then the assignment
$$u_{X,k,q}\colon F_X^{k,q} \, \to \, \pi_0^q\big(V^{n-k}(\omega_X)_0\big),\quad \big(f\colon X\to V \big) \, \mapsto \, \pi_* \widetilde{f}^*\Pic^0(\widetilde V)$$ 
  is well-defined and yields a bijection of sets.
\end{theorem}

\section{Proofs}\label{secproof}
 \subsection{Proof of Theorem \ref{mainthm}} 
 \begin{proof}[Proof of Theorem \ref{mainthm}]
 Set $n=\dim X = \dim Y$
 and let $p_1$ and $p_2$ be the projections from $X\times Y$ onto   $X$ and $Y$, respectively.  
 The equivalence $\Phi$ is of Fourier--Mukai type, that is
 $\Phi(-) \simeq \Phi_{\sE}(-) = \rr p_{2*}\big( p_1^*(-) \stackrel{\rl}{ \otimes } \sE \big)$ for some object $\sE$ in $\rd(X\times Y)$. 
By \cite[Lemma 3.1]{popa+schnell:invariance} the induced Rouquier isomorphism 
$$F_{\sE} \; \colon \; \Aut^0(X)\, \times \, \Pic^0(X) \; \longrightarrow \; \Aut^0(Y) \, \times \, \Pic^0(Y)$$  acts as
\begin{equation}\label{rouquieraction}
F_{\sE}(\varphi,\alpha) \; = \; (\psi,\beta)\; \Longleftrightarrow \;
p_1^*\alpha \,\otimes \, (\varphi\times {\rm id}_Y)^*\sE \; \simeq \; p_2^*\beta \, \otimes \, ({\rm id}_X \times \psi)_*\sE.
\end{equation}
Moreover, by Theorem \ref{lopo} and Remark \ref{hodge}, $F_{\sE}$ induces an isomorphism of algebraic sets $V^{n-1}(\omega_X)_0\simeq V^{n-1}(\omega_Y)_0$. 
By \S\ref{secpreview} and Theorem \ref{fibrAlbprim} it follows that 
the composition $$\mu_g \; \stackrel{{\rm def}}{=} \; \big( u_{Y,1,g}^{-1} \, \circ \, F_{\sE} \, \circ \, u_{X,1,g} \big) \colon F^{g}_X \; \to \;  F^g_Y$$ defines a bijection of sets for any integer $g\geq 2$   such that,  if  $\mu_g(f\colon X\to C)=(h\colon Y\to D)$, then 
\begin{equation}\label{rouqset}
F_{\sE}\big({\rm id}_X, f^*\Pic^0(C)\big) \; = \; \big({\rm id}_Y,h^*\Pic^0(D)\big).
\end{equation}
Thus we only need to show that the curves  $C$ and $D$ are isomorphic. This is  achieved  by constructing a complex in $\rd(C\times D)\stackrel{{\rm def}}{=} D^b \big(\mathcal{C}oh(C\times D) \big)$  
whose support dominates both $C$ and $D$.

The equations \eqref{rouquieraction} and  \eqref{rouqset} define an isomorphism of abelian varieties $\zeta\colon \Pic^0(C)\, \rightarrow \, \Pic^0(D)$ such that
\begin{equation}\label{rouq1}
 p_1^*f^*\alpha \, \otimes \, \sE \; \simeq \; p_2^*h^*\zeta(\alpha) \, \otimes \, \sE \quad \mbox{ for all }\quad \alpha \in \Pic^0(C).
\end{equation}
We fix   a sufficiently ample line 
bundle $L$ on $X\times Y$ and define
$$\sE' \; \stackrel{\rm def}{=} \; \rr (f\times h)_*(\sE\otimes L).$$
Moreover, we  denote by $t_1$ and $t_2$ the projections from $C\times D$ onto the first and second factor, respectively. 
An application of the projection formula yields isomorphisms 
$$\rr(f\times h)_*\big( p_1^*f^*\alpha \otimes  \sE \otimes L \big)\, \simeq \,\rr(f\times h)_*\big((f\times h)^*t_1^*\alpha\otimes \sE \otimes L\big) \; \simeq \;
t_1^*\alpha \, \otimes \, \sE ', $$
$$\rr(f\times h)_*\big( p_2^*h^*\zeta(\alpha) \otimes  \sE \otimes L \big)\, \simeq \,\rr(f\times h)_*\big((f\times h)^*t_2^*\zeta(\alpha)\otimes \sE \otimes L\big) \; \simeq \;
t_2^*\zeta(\alpha) \, \otimes \, \sE ', $$
so that 
\begin{equation}\label{tensoriso}
t_1^*\alpha \, \otimes \, \sE ' \; \simeq \; t_2^* \zeta(\alpha)  \otimes  \sE ' \quad \mbox{ for all } \quad \alpha \in \Pic^0(C).
\end{equation}
Let  $i_b\colon t_2^{-1}(b) \hookrightarrow C\times D$ be the closed immersion of the fiber $t_2^{-1}(b)\simeq C$, and set  $\sE'_b \,  \stackrel{\rm def}{=}  \, \mathbf{L}i_b^*\sE'$.
By restricting the isomorphisms   \eqref{tensoriso} to every
 fiber of $t_2$,
 we obtain new isomorphisms:
\begin{equation}\label{each2}
\sE'_b \, \otimes \, \alpha \; \simeq \; \sE'_b\, \quad\mbox{ for  \; all }\quad
\alpha \in \Pic^0(C) \quad \mbox{and} \quad b\in D.
\end{equation}
We denote by $H^q(\sF)$ the $q$-th cohomology sheaf of a complex  $\sF$, and 
define the support of $\sF$ by  $\Supp(\sF)=\bigcup_q \Supp \big(H^q(\sF) \big)$.
 We endow ${\rm Supp} \big(\sF \big)$ with its reduced scheme structure.  
By taking cohomology in \eqref{each2}, we find   $H^q(\sE'_b ) \otimes \alpha \simeq H^q(\sE'_b)$ for all $q\in \rz$, $\alpha \in \Pic^0(C)$ and $b\in D$.
Thus by \cite[Lemma 3.3]{mukai:equivalence} we obtain
$$\dim \Supp \big( H^q(\sE_b') \big)  = \dim \Supp \big( (j_{C})_* H^q(\sE_b') \big) \leq 0$$ for all $q\in \rz$ and  $b \in D$ (here 
$j_C\colon C \hookrightarrow \Pic^0(C)$ denotes the Abel--Jacobi embedding). 
We conclude   $\dim \Supp \big(\sE_b' \big) \leq 0$ for all  $b\in D$ since $\sE_b'$ is a bounded complex. 
Moreover, from the isomorphisms 
 $${\rm Supp}\big(\sE_b'\big)\; = \; {\rm Supp}\big(\sE'\big)\,\cap \, t_2^{-1}(b) \quad \mbox{ for all }\quad b\in D$$
(\emph{cf}. \cite[Lemma 3.29]{huybrechts:fouriermukai}),
 we conclude  that  $t_2\colon \Supp(\sE')\to D$ is a finite morphism and 
 $\dim {\rm Supp}\big(\sE'\big)\, \leq \, 1.$ Also, in a complete analogy, also     $t_1 \colon  \Supp(\sE')\to C$ is a finite morphism.

Now consider the following commutative diagram:
$$
\centerline{ \xymatrix@=28pt{
 \Supp\big(\sE\otimes L\big) \ar[r]^{\;\;\;\;\;\;\;\;p_1} \ar[d]^{f\times h} & X\ar[d]^{f} \\
\Supp\big(\sE'\big) \ar[r]^{\;\;\;\; t_1} & C. \\}} 
 \noindent
  $$
 The left arrow of the above diagram is well-defined thanks to  the following claim.
 
 \begin{claim}\label{claim1}
The restriction of $f\times h$ to ${\rm Supp} \big(\sE \otimes L\big) = {\rm Supp} \big(\sE\big)$ defines a  morphism  $( f\times h ) \colon {\rm Supp} \big(\sE \big) \to {\rm Supp}\big(\sE'\big)$ of algebraic sets.
\end{claim}
\begin{proof}
As $L$ is sufficiently ample, on the one hand we have  the following  vanishings 
$$R^p(f\times h)_* \big(H^q(\sE)\otimes L \big) \; = \; 0\quad \mbox{ for all } \quad p>0 \quad \mbox{ and } \; q\in \rz.$$ 
On the other hand there are surjections $$(f\times h)^*(f\times h)_* \big(H^q(\sE)\otimes L \big) \; \twoheadrightarrow \; H^q(\sE)\otimes L\quad \mbox{ for all }\;\;  q.$$ 
Therefore, by the degeneration of the spectral sequence 
$$E^{p,q}_2 \; := \; R^p(f\times h)_*(H^q(\sE)\otimes L) \; \Rightarrow \; R^{p+q} (f\times h)_*(\sE\otimes L),$$ 
we obtain the following isomorphisms for all $q$:
$$H^q(\sE') \;  = \; R^q(f\times h)_* \big(\sE\otimes L \big) \; \simeq \; (f\times h)_* \big(H^q(\sE)\otimes L \big)$$ and  
$$\Supp \big((f\times h)_* \big(H^q(\sE)\otimes L \big) \big) \; = \; (f\times h)\big(\Supp \big(H^q(\sE)\otimes L \big)\big).$$ 
We conclude   
$\Supp \big(H^q(\sE') \big) \; \simeq \; (f\times h)\big(\Supp \big(\,H^q(\sE)\otimes L \big) \big),$ from which we deduce $${\rm Supp}\,\big(\sE'\big)\, = \, \bigcup_q \Supp\big(H^q(\sE') \big) \, \simeq \, \big(f\times h\big)\Big(\bigcup_q \Supp\big(H^q(\sE) \otimes L \big) \Big) \;=\; (f\times h)\big({\rm Supp}\,\big(\sE \otimes L\big)\big).$$
\end{proof} 
  
 
 To conclude the proof, 
 we observe   the projection $p_1$ is surjective with connected fibers (\emph{cf}. \cite[Lemmas 6.4 and 6.11]{huybrechts:fouriermukai}). Thus,   $t_1\colon \Supp\big(\sE'\big)\rightarrow C$ is surjective with connected fibers.
It follows that $\Supp \big(\sE' \big)$ is irreducible, and that $t_1$ is an  isomorphism by \cite[Corollary 4.4.3]{liu}.
In complete analogy, we have $\Supp\big(\sE'\big) \simeq D$.  
\end{proof}
 
\begin{rmk}
As noted in \cite{caucci+lombardi:irregular}, \cite{caucci+lombardi+pareschi:invariants2}, the argument of 
the proof of Theorem \ref{mainthm} can be employed to study the 
derived invariance of the finite part of the Stein factorization of the composition between the Iitaka fibration and the 
Albanese map of its base variety.
\end{rmk} 
%
%

	\subsection{Proof of Theorem \ref{mainthm2}}
	\begin{proof}[Proof of Theorem \ref{mainthm2}]

As in the proof of Theorem \ref{mainthm} we have  $\Phi(-) \simeq \Phi_{\sE}(-) =  \rr p_{2*}\big( p_1^*(-) \stackrel{\rl}{ \otimes } \sE \big)$ 
where $p_1,p_2$ are  the projections from $X\times Y$ onto   $X$ and $Y$, respectively, and $\sE$ is an object in 
$\rd(X\times Y)$.
Let $(v \colon X \to S)  \in F_X^{2,q}$ be the isomorphism class of 
a $\chi$-positive $2$-higher irrational pencil $v\colon X \to S$ with $q(S)=q$.
Moreover, let  $(\pi_X \colon \widetilde X \to X, \, \widetilde v \colon \widetilde X \to \widetilde S)$ be a non-singular representative of $v$.
By Corollary \ref{corfibration}, the abelian variety $Z\stackrel{{\rm def}}{=}\pi_{X*} \widetilde v^* \Pic^0(\widetilde S)$ is  an irreducible component of $V^{n-2}(\omega_X)_0$ which  does not depend on the choice of the non-singular representative $(\pi_X, \widetilde v)$.
By Theorem \ref{lopo}, the Rouquier isomorphism $F_{\sE}$ induced by  $\Phi_{\sE}$  yields an irreducible component 
 $Z' \stackrel{\rm def}{=} F_{\sE}({\rm id}_X,Z)\subset V^{n-2}(\omega_Y)_0$. This in turn induces,  by means of Lemma \ref{5.1}, a fibration 
 $w\colon Y \to T$ onto a normal projective variety $T$ such that $0< \dim T \leq 2$ and $q(T)=q\geq 2$.
If $(\pi_Y \colon \widetilde Y \to Y, \,  \widetilde w \colon \widetilde Y \to \widetilde T)$  is a non-singular representative of $w$, then  the following properties hold:  
 $(i)$  there exists a morphism 
$b\colon T\to \Alb(\widetilde T)$ finite onto its image such that the composition $\widetilde T \to T \stackrel{b}{\to} \Alb(\widetilde T)$ equals 
$alb_{\widetilde T}$, $(ii)$ 
$\chi(R^{n-2}\widetilde w_*\omega_{\widetilde Y} )>0$, and $(iii)$ $Z'= \pi_{Y*} \widetilde  w^*\Pic^0(\widetilde T) = \pi_{Y*} \widetilde w^*V^0(R^{n-2}\widetilde w_*\omega_{\widetilde Y})$.

We claim   $\dim T=2$.  If, by contradiction, we   had  $\dim T=1$, 
then the following facts would be true: $(i)$ $\widetilde T \simeq T$, $(ii)$ $V^0(\omega_T)=\Pic^0(T)$,\footnote{Any smooth projective curve of genus $g\geq 2$ satisfies $V^0(\omega_C)=\Pic^0(C)$. 
On the contrary, in higher dimension  there are varieties $Z$ of maximal Albanese dimension such that  $q(Z)>\dim Z$ and $V^0(\omega_Z)\subsetneq \Pic^0(Z)$.  For instance,  isotrivial elliptic surfaces $S$  fibered over curves of genus at least two  provide counterexamples for any $q(S)\geq 3$. This is the reason why  the arguments of this  proof do not extend to fibrations in $F_X^{k,q}$ with $k\geq 3$.} and $(iii)$  $w^*\Pic^0(T)=Z'\subset V^{n-1}(\omega_Y)_0$  (\emph{cf}. \cite[Lemma 6.3]{lombardi:invariants}). Hence, by Remark \ref{hodge} and Theorem \ref{lopo},  the component  $Z$ is contained in $V^{n-1}(\omega_X)_0$,  and by Lemma \ref{5.1}    it  induces an irrational pencil $u\colon X\to B$ of genus $q$ such that $Z=u^*\Pic^0(B)$. By Lemma \ref{gamma2}, we  get an isomorphism between $B$ and $S$ which is impossible. We conclude    $T$ is a surface, and  the fibration $w \colon Y \to T$ defines   a class in  $F_Y^{2,q}$ (\emph{cf}. Corollary \ref{corkollar}). Moreover, the argument also shows    the assignments $u_{X,2,q} \colon F_X^{2,q} \to  \pi_0^q\big(V^{n-2}(\omega_X)_0\big) $ and $u_{Y,2,q} \colon F_Y^{2,q} \to \pi_0^q\big(V^{n-2}(\omega_Y)_0\big)$  are surjective, and hence bijective
by Proposition \ref{prop:bijinj}. Finally,  by Theorem \ref{lopo},  the composition 
$$\nu_q \, \stackrel{{\rm def}}{=} \, \big( u_{Y,2,q}^{-1}  \, \circ \,F_{\sE} \, \circ \, u_{X,2,q} \big) \colon F^{2,q}_X \; \to \; F^{2,q}_Y$$ defines a bijection of sets for all $q\geq 2$ such that, if $(v \colon X\to S)\in F_X^{2,q}$ and $\nu_q(v)=(w\colon Y\to T)$, then 
\begin{equation}\label{isor2}
F_{\sE}\big({\rm id}_X, \pi_{X*} \widetilde v^*\Pic^0(\widetilde S)\big) \; = \;  \big({\rm id}_Y, \pi_{Y*} \widetilde w^*\Pic^0(\widetilde T)\big).
\end{equation}
 In the rest of the proof we will show   the surfaces $S$ and $T$ are isomorphic.
 
We consider    non-singular representatives of $v$ and $w$ as follows:
$$
 \xymatrix@=28pt{
\widetilde X \ar[r]^{\pi_X} \ar[d]^{\widetilde v} & X\ar[d]^{v} \\
\widetilde S \ar[r]^{\rho_S} & S }\quad \quad\quad \quad \quad 
\xymatrix@=28pt{
\widetilde Y \ar[r]^{\pi_Y} \ar[d]^{\widetilde w} & Y\ar[d]^{w} \\
\widetilde T \ar[r]^{\rho_T}  & T. \\} 
   $$ 
 By    Remarks \ref{rmkalb} and \ref{rmkresolution}, the surfaces $S$ and $T$ are equipped with   morphisms $a\colon S \to \Alb(\widetilde S)$ and 
$b\colon T \to \Alb(\widetilde T)$, respectively,  finite onto their images, such that $a\circ \rho_S = alb_{\widetilde S}$ and $b\circ \rho_T = alb_{\widetilde T}$.
We set  $\pi = \pi_X\times \pi_Y$ and denote by  $\xi\colon \Pic^0(\widetilde S)\to \Pic^0(\widetilde T)$  the isomorphism induced by \eqref{isor2}. By Proposition \ref{picbir},  $\xi$ induces an isomorphism $\Pic^0(S)\simeq \Pic^0(T)$, which we continue to denote, with a slight abuse of notation, by $\xi$.
In this way we have $\xi \circ \rho_S^* = \rho_T^* \circ \xi$.
 Finally, let $\widetilde{p}_1$ and $\widetilde{p}_2$ be  the projections from $\widetilde{X} \times \widetilde{Y}$ onto the first and second factor, respectively. By the  equality \eqref{isor2} and the description of the action of $F_{\sE}$ in  \eqref{rouquieraction}, we obtain  a collection of isomorphisms:
\begin{equation}\label{isopullback2}
\pi_* \, \widetilde{p}_1^* \, \widetilde{v}^*\alpha \, \otimes \, \sE \; \simeq \; \pi_* \, \widetilde p_2^* \, \widetilde w^*\xi (\alpha) \, \otimes \, \sE \quad \mbox{ for all }\quad \alpha \in \Pic^0(\widetilde S).
\end{equation}
By Proposition \ref{picbir},   the isomorphisms \eqref{isopullback2} are equivalent to the following isomorphisms:
$$\pi_* \, \widetilde{p}_1^* \, \widetilde{v}^*\rho_S^*\gamma \, \otimes \, \sE \; \simeq \; \pi_* \, \widetilde p_2^* \, \widetilde w^*\rho_T^*\xi (\gamma) \, \otimes \, \sE \quad \mbox{ for all }\quad \gamma \in \Pic^0(S).$$
Denote now by $t_1$ and $t_2$ the projections from $S\times T$ onto $S$ and $T$, respectively.
By tensoring the previous isomorphisms by a sufficiently ample line bundle $L$ on $X\times Y$, and by pushing them forward to $S\times T$, we obtain new isomorphisms:
\begin{equation}\label{tensoriso2}
t_1^*\gamma \, \otimes \, \sE' \; \simeq \; t_2^*\xi(\gamma) \, \otimes \, \sE' \quad \quad \mbox{for all}\quad \gamma \in \Pic^0(S),
\end{equation} where 
$$\sE' \; \stackrel{{\rm def}}{=} \; \rr (v\times w)_* \big( \sE \otimes L\big).$$ 
 We set $$\sE''\stackrel{{\rm def}}{=}\rr(a\times b)_*\sE' = (a\times b)_*\sE'$$ and denote by $q_1$ and $q_2$  the two projections from $\Alb(\widetilde S)\times \Alb(\widetilde T)$ onto the first and second factor, respectively. 
By pushing the isomorphisms \eqref{tensoriso2} forward to $\Alb(\widetilde S)\times \Alb(\widetilde T)$, we obtain new isomorphisms:
$$ (a\times b)_* \big(t_1^*\gamma \otimes \sE' \big) \; \simeq \;  (a \times b)_* \big(t_2^*\xi(\gamma) \otimes \sE' \big) \quad \quad \mbox{ for all } \quad \gamma \in \Pic^0(S).$$
By noting that every $\gamma\in \Pic^0(S)$  can be written as $\gamma = a^* \beta$ for a unique $\beta \in \Pic^0 \big(\Alb(\widetilde S) \big)$, the projection formula yields   
(again with a slight abuse of notation)\footnote{We continue to denote by $\xi \colon \Pic^0 \big(\Alb(\widetilde S) \big) \to \Pic^0 \big(\Alb(\widetilde T) \big)$ the morphism induced by 
$\xi \colon \Pic^0(S ) \to \Pic^0(T)$. With this notation we find $\xi \circ a^* = b^* \circ \xi$.
}
\begin{equation}\label{iso23}
q_1^*\beta \, \otimes \, \sE'' \; \simeq \; q_2^*\xi(\beta) \, \otimes  \, \sE'' \quad \quad \mbox{ for all }\quad \beta \in \Pic^0 \big(\Alb(\widetilde S) \big).
\end{equation}
Denote   by $\sE''_s \stackrel{{\rm def}}{=} \sE''_{|q_1^{-1}(s)}$ the derived restriction of $\sE''$ to the fiber  $q_1^{-1}(s)$.
By restricting the isomorphisms \eqref{iso23} to every fiber of $q_1$, we have isomorphisms 
$$\sE''_s \; \simeq \; \sE''_s\otimes \xi(\beta)  \quad \quad \mbox{ for all } \quad \beta \in \Pic^0\big(\Alb(\widetilde S) \big) \quad \mbox{ and  \quad } s \in \Alb(\widetilde S).$$
As the complex  $\sE''_s$ has only a finite number of non-zero cohomology sheaves, Mukai's lemma \cite[Lemma 3.3]{mukai:equivalence} implies   $\dim \Supp \big(\sE_s''\big) \leq 0$ for all $s$. Similarly, we also have $\dim \Supp\big( \sE''_{|q_2^{-1}(t)} \big) \leq 0$ for all $t\in \Alb(\widetilde T)$. 
By \cite[Lemma 3.29]{huybrechts:fouriermukai} the  projections  $$q_1\colon \Supp \big(\sE'')\to \Alb(\widetilde S)\quad \mbox{ and }\quad  
q_2\colon \Supp \big(\sE'')\to \Alb(\widetilde T)$$ are finite morphisms and  $\dim \Supp\big(\sE''\big)\leq 2$.
Moreover, by the argument of Claim \ref{claim1}, there are commutative diagrams 
$$
 \xymatrix@=28pt{
 \Supp \big(\sE'  \big) \ar[r]^{\,\;\;\;\;\;t_1} \ar[d]^{a \times b} & S\ar[d]^{a} \\
\Supp \big(\sE''\big) \ar[r]^{\;\;\;\; q_1} & \Alb(\widetilde S)\\} \quad \quad \quad \quad 
 \xymatrix@=28pt{
 \Supp \big(\sE'  \big) \ar[r]^{\,\;\;\;\;\;t_2} \ar[d]^{a \times b} & T\ar[d]^{b} \\
\Supp \big(\sE''\big) \ar[r]^{\;\;\;\; q_2} & \Alb(\widetilde T)\\}
  $$
from which we also  infer that also the morphisms  $t_1$ and $t_2$ are finite.
In order to conclude the proof,  we consider  the following commutative diagrams:
$$
 \xymatrix@=28pt{
 \Supp \big(\sE  \big) \ar[r]^{\,\;\;\;\;\;\,\;\;\;\;\;p_1} \ar[d]^{v\times w} & X\ar[d]^{v} \\
\Supp \big(\sE'\big) \ar[r]^{\;\;\;\; t_1} & S\\} \quad \quad \quad \quad 
 \xymatrix@=28pt{
 \Supp \big(\sE  \big) \ar[r]^{\,\;\;\;\;\;\,\;\;\;\;\; p_2} \ar[d]^{v\times w} & Y\ar[d]^{w} \\
\Supp \big(\sE'\big) \ar[r]^{\;\;\;\; t_2} & T.\\}
  $$
As $p_1 \colon \Supp \big(\sE  \big)  \to X$ and $p_2 \colon \Supp \big(\sE  \big)  \to Y$ are surjective   with connected fibers 
(\emph{cf}. \cite[Lemmas 6.4 and 6.11]{huybrechts:fouriermukai}),    the morphisms  $t_1 \colon  \Supp \big(\sE'  \big)  \to S$ and $t_2\colon  \Supp \big(\sE'  \big)  \to T$ also are surjective   with  connected fibers. Therefore $\Supp\big(\sE'\big)$ is irreducible and   $S\simeq \Supp\big(\sE'\big) \simeq T$.  
\end{proof}

In the second part of the previous proof the dimensions of the varieties $S$ and $T$ did not play a role. In fact the argument extends to fibrations of  higher-dimensional bases as follows.

\begin{prop}\label{deriso}
Suppose $\Phi_{\sE} \colon \rd(X) \stackrel{\sim}{\longrightarrow} \rd(Y)$ is an equivalence of triangulated categories and let $F_{\sE}$ be the induced Rouquier isomorphism.  
Let $(f \colon X \to V)\in F_X^{k,q}$ and $(h \colon Y \to W)\in F_Y^{k,q}$ be $\chi$-positive higher irrational  pencils with $q\geq k\geq 1$. 
If $$F_{\sE}\big({\rm id}_X , \pi_{X*}\widetilde f^* \Pic^0(\widetilde V) \big) \,  = \,  \big({\rm id}_Y , \pi_{Y*}\widetilde h^* \Pic^0(\widetilde W) \big)$$ for some non-singular representatives 
$(\pi_X \colon \widetilde X \to X,\; \widetilde f \colon \widetilde X \to \widetilde V)$ and $(\pi_Y \colon \widetilde Y \to Y, \; \widetilde h \colon \widetilde Y \to \widetilde W)$  of $f$ and $h$, respectively, then  $V$ and $W$ are isomorphic.
\end{prop}

\subsection{Proof of Theorem \ref{introb}} 
 Let $X$ be a smooth projective variety of dimension $n$. Before proving Theorem \ref{introb}, we recall   the integer  $b_{\chi>0}(X) \in \{0,\ldots, n-1\}$ is defined as  the minimal dimension of a  base variety of a $\chi$-positive $k$-higher irrational pencil with  $k \in \{ 1 , \ldots , n - 1 \}$. Moreover, we declare  $b_{\chi>0}(X) = 0$ if and only if $X$ does not admit any $\chi$-positive $k$-higher irrational pencil for all  $k \in \{ 1 , \ldots , n - 1 \}$. See also \cite[\S 4]{caucci+pareschi:invariants}.

\begin{prop}\label{propdimless}
Let $0<k<n$ be an integer and let $Z$ be  an irreducible component of $V^{n-k}(\omega_X)_0$ of dimension $\dim Z \geq k$. Denote by 
$f\colon X \to V$ the fibration induced by $Z$ as in Lemma \ref{5.1}. If $\dim V<k$, 
 then $X$ admits a fibration $p' \colon X\to U$ in $F_X^{k',q'}$ for some integers $0<k'< k$ and $2\leq q'\leq \dim Z$. Moreover,  there exists a dominant rational map $\gamma \colon V \dashrightarrow U$ such that $p'=\gamma \circ f$.

\end{prop}
\begin{proof}
Let $\widetilde f \colon \widetilde X \to \widetilde V$ be a non-singular representative of $f$, and note that 
$q(\widetilde V) = q(V) = \dim Z \geq k > \dim \widetilde V$.
Hence  $alb_{\widetilde V}$ is not surjective and,
 by \cite[Proposition 2.2]{ein+laz:theta},  $\sO_{\widetilde V}$ is not an isolated point in $V^0(\omega_{\widetilde V})$. We conclude that   $\dim V^0(\omega_{\widetilde V})_0>0$. 
If $V^0(\omega_{\widetilde V})=\Pic^0(\widetilde{V})$, then  by Lemma \ref{lemchi} we   have
$\chi(\omega_{\widetilde V})>0$ so that     $f$ defines a class of fibrations in $F_X^{\dim V,q}$ with $q=q(V) = \dim Z$. 
Suppose now that $V^0(\omega_{\widetilde V}) \neq \Pic^0(\widetilde{V})$ and $Z'' \subset V^0(\omega_{\widetilde V})_0$ is a positive-dimensional irreducible component of codimension $0<j\stackrel{{\rm def}}{=}\codim Z''<q(\widetilde{V})$.  Then $Z''$ is also an irreducible component of $V^j(\omega_{\widetilde V})_0$ (\emph{cf}. \cite[Corollary 3.3]{pareschi:standard}). Therefore, by Lemma \ref{5.1}, there exists a further fibration $f'\colon \widetilde{V} \to V'$ such that $$0 \; < \; \dim V' \; \leq \; \dim V -j \; < \; k-j.$$
Moreover, 
if $\widetilde{f'} \colon \widetilde{\widetilde V} \to \widetilde {V'}$ is a non-singular representative of $f'$, then  $\chi(R^j\widetilde{f'}_*\omega_{\widetilde{\widetilde{V}}})>0$  and  $Z''=\widetilde {f'}^*\Pic^0(\widetilde{V'})$. 
Moreover, by the previous argument, we still  have $\dim V^0(\omega_{\widetilde{V'}})_0>0$ since 
$$q(\widetilde {V'}) \, = \, \dim Z'' \, = \, q(\widetilde V) - j \, = \, q(V) - j \, \geq \, \dim V -j \, = \, k-j \, > \, \dim \widetilde{V'}.$$ 
Hence we obtain a fibration $l\colon \widetilde X \to V'$ with the same properties of  $f$, but satisfying  $\dim V'<\dim V$.
Proceeding in this way, we will  eventually construct   a fibration $p\colon X^* \to U$ in $F_{X^*}^{k',q'}$ where $X^*$ is a smooth projective variety birational to $X$ with  $0<k' = \dim U <k$ and $\dim U\leq q' \stackrel{\rm def}{=} q(U) \leq q(V)$. However,  if $U$ were a curve, then we must have $q'\geq 2$, as,  by \cite[Corollaire 2.3]{beauville:h1}, the $(n-1)$-th non-vanishing locus attached to the canonical bundle  does not contain any component of dimension one passing through the origin. 
By Proposition \ref{extension} the 
rational dominant map $X\dashrightarrow U$ extends to a morphism $p'$ such that the composition $X^*\to X \stackrel{p'}{\to } U$ equals $p$.
The second statement   is a  consequence of the construction.
\end{proof}

   The following proposition is inspired by \cite[Claim 4.9]{caucci+pareschi:invariants}.  

  \begin{prop}\label{propb}
We have $b_{\chi>0}(X) > 0$ if and only if there exists an index $i\in \{1,\ldots , n-1\}$ such that $\dim V^i(\omega_X)_0 \geq  n-i$.   
  Moreover, if $b_{\chi>0}(X)>0$, then $n-b_{\chi>0}(X)$ is the largest integer $0 < i < n$ such that $\dim V^i(\omega_X)_0 \geq n-i$.
    \end{prop}

\begin{proof}

Suppose  there exists an irreducible component $Z \subset V^i (\omega_X)_0$ of dimension $\dim Z \geq n-i$ for some $i>0$.
Let $f\colon X \to V$ be the fibration induced by $Z$, as in  Lemma \ref{5.1}, with $\dim V \leq n-i$.
We distinguish two cases: $\dim V = n-i$ and $\dim V < n-i$. If $\dim V= n-i$, then, by Corollary \ref{corkollar}, we have  $f \in F_X^{n-i, \dim Z}$.
On the other hand, if $\dim V < n-i$, then, by Proposition \ref{propdimless}, we can construct a fibration in $F_X^{k',q'}$ with $0<k'<n-i$ and $2\leq q'\leq \dim Z$. 
In both cases we have $b_{\chi>0}(X)>0$. 

Conversely, if $f\colon X \to V$ is a $\chi$-positive $b$-higher irrational pencil with $b=b_{\chi>0}(X)>0$, then, by Corollary \ref{corfibration}, $\pi_* \widetilde f^*\Pic^0(\widetilde V)$ is an irreducible component of $V^{n-b}(\omega_X)_0$ of dimension $q(V)$
(as usual $(\pi,\widetilde f)$ denotes a non-singular representative of $f$). Since $\widetilde V$ is of maximal Albanese dimension, we have $q(V)\geq \dim V=b$.
 The second claim is proved similarly.
%
\end{proof}


\begin{prop}\label{proponeb}
Suppose   $b\stackrel{\rm def}{=}b_{\chi>0}(X)>0$ and let $q\geq b$ be an integer.
Then the assignment
$$u_{X,b,q} \colon F_X^{b, q} \; \to \; \pi_0^q \big(V^{n-b} (\omega_X)_0\big),\quad \big(f\colon X\to V \big) \, \mapsto \, \pi_* \widetilde{f}^*\Pic^0(\widetilde V)$$ as defined in Proposition \ref{prop:bijinj} is well-defined and it defines  a bijection of sets.
\end{prop}

\begin{proof}
Proposition \ref{prop:bijinj} shows   the map $u_{X,b,q}$ is well-defined and injective.
To show   $u_{X,b,q}$ is also surjective, we denote by $Z\subset V^{n-b}(\omega_X)_0$  an arbitrary irreducible component of dimension $q$. By Lemma \ref{5.1} the component $Z$ induces a fibration $f\colon X \to V$ onto a normal projective variety $V$  with $\dim V\leq b$. If $\dim V<b$, then by Proposition \ref{propdimless} we can construct a fibration in $F_X^{k',q'}$ with $0<k'<b$ and $2\leq q'\leq q$. This contradicts the definition of $b$. Therefore $\dim V = b$, and $f\in F_X^{b,q}$ by Corollary \ref{corkollar}.

\end{proof}

\begin{proof}[Proof of Theorem \ref{introb}.]
By Proposition \ref{propb}, the integers $b_{\chi>0}(X)$ and $b_{\chi>0}(Y)$  only depend on the dimensions of the non-vanishing loci $V^i(\omega_X)_0$ and 
$V^i(\omega_Y)_0$ for all $i\geq 0$.
Therefore,  by Theorem \ref{lopo}, we have   $b_{\chi>0}(X) = b_{\chi>0}(Y)$. 
We set $b= b_{\chi>0}(X) = b_{\chi>0}(Y)$.  
Denote now by $F_{\sE}$  the Rouquier isomorphism induced by the equivalence $\Phi = \Phi_{\sE} \colon \rd(X) \to \rd(Y)$. 
By  Proposition \ref{proponeb} and Theorem \ref{lopo}
 the following assignments 
$$\sigma_q \; \stackrel{{\rm def}}{=} \; \big( u^{-1}_{Y,b,q} \,  \circ \, F_{\sE} \,  \circ \, u_{X,b,q} \big) \colon F_X^{b,q} \; \to \; F_Y^{b,q}$$
are  bijections  for all  $q\geq b$. Moreover, if $\sigma_q (f \colon X \to V ) = (h \colon Y \to W)$, then we have
$$F_{\sE} \big( {\rm id}_X , \pi_{X*} \widetilde f^* \Pic^0({\widetilde V}) \big) \; = \; \big(  {\rm id}_Y , \pi_{Y*} \widetilde h^* \Pic^0({\widetilde W} )\big)$$
where  $(\pi_X \colon \widetilde X \to X, \; \widetilde f \colon \widetilde X \to \widetilde V)$ and 
$(\pi_Y \colon \widetilde Y  \to Y, \;  \widetilde h \colon \widetilde Y \to \widetilde W)$ are  non-singular representatives of $f$ and $h$, respectively.
The proof that $V$ and $W$ are isomorphic follows by Proposition \ref{deriso}.
\end{proof}

\section{Applications to the fundamental group}\label{secapplications}

The fundamental group $\pi_1(X)$ of a  smooth projective variety is not in general a derived invariant, as shown for instance by Schnell in \cite{schnell:fundamental}. Nevertheless, as an application of Theorem \ref{mainthm}, we deduce the derived invariance of the following property. In this section $X$ and $Y$ denote two smooth projective varieties.
\begin{cor}\label{introcor1}
  If $\rd(X) \simeq \rd(Y)$ and $g\geq 2$ is an integer,  then   $\pi_1(X)$ admits a surjective homomorphism onto the fundamental group of a compact Riemann surface of genus $g$ if and only if $\pi_1(Y)$ does. 
\end{cor}

\begin{proof}
In the Appendix of \cite{catanese:moduli}, Beauville proves that, for any $g\geq 2$,  there exists a surjective homomorphism $\pi_1(X)\rightarrow \Gamma_g$ onto the fundamental group of a compact Riemann surface of genus $g$ if and only if $X$ admits an irrational pencil of genus greater or equal to $g$. The  corollary  follows by Theorem \ref{mainthm}.
\end{proof}
%

\begin{rmk} 
In \cite[Theorem 1.10]{catanese:moduli} (\emph{cf}. also \cite[Theorem 1.1]{kot:paper}), the author relates the existence of irrational pencils  of genus $g\geq 2$ to the existence  of maximal isotropic subspaces in $H^1(X,\rc)$  (we recall  subspace $U\subset H^1(X,\rc)$ is   \emph{isotropic} if the map $U\wedge U \to H^2(X,\rc)$ is the trivial map). As an  application of  Theorem \ref{mainthm}, we deduce   $X$ admits a maximal isotropic subspace $U\subset  H^1(X,\rc)$ of dimension $g\geq 2$ if and only if $Y$ does.  More refined statements along  these lines can be formulated by means of \cite[Theorem 2.25]{catanese:moduli}.
\end{rmk}

A further application of Theorem \ref{mainthm} involves the 
cup product map $$\varphi_{X}\colon H^1(X,\rc)\wedge H^1(X,\rc)\to H^2(X,\rc).$$
In view of Castelnuovo--De Franchis' theorem, Theorem \ref{mainthm}  implies that, if $\rd(X) \simeq \rd(Y)$, then
$\varphi_X$ is injective on pure forms of type $\omega_1\wedge \omega_2$   if and only if 
$\varphi_Y \colon H^1(Y,\rc)\wedge H^1(Y,\rc)\to H^2(Y,\rc)$ has the same property. A more general statement holds if the geometric genus $p_g(X)$ is equal to one.
\begin{cor}\label{introcor3}
Suppose    $\rd(X) \simeq \rd(Y)$ and $p_g(X) =1$. Then 
$\varphi_X$ is injective if and only if 
$\varphi_Y$ is injective.
\end{cor}

\begin{proof}
By  \cite[Theorem 1.3]{chen+jiang+tian:irregular}, the map $\varphi_X$ is injective if and only if $X$ does not carry  a irrational pencil of genus two. As $p_g(X) = p_g(Y)$, the corollary follows from Theorem \ref{mainthm}.
\end{proof}
%
%

%

\section{Further results}\label{secfurther} 
In this section we describe the behavior of $\chi$-positive $3$-higher irrational pencils of $D$-equivalent fourfolds. 
 We begin   with a general result. 
 
 We denote by $X$ and $Y$ two smooth projective varieties of dimension $n\geq 2$. Suppose    $\Phi_{\sE} \colon \rd(X) \stackrel{\sim}{\longrightarrow} \rd(Y)$ is an equivalence of triangulated categories and let   $F_{\sE}$ be the induced Rouquier isomorphism. 
 Moreover let  $f\colon X \to V$ be a $\chi$-positive $k$-higher irrational pencil with $0<k<n$ and irregularity $q=q(V) \geq k$. If $(\pi \colon \widetilde X \to X,  \; \widetilde f \colon \widetilde X \to \widetilde V)$ is a non-singular representative of $f$, then $Z\stackrel{{\rm def}}{=} \pi_* \widetilde f^* \Pic^0(\widetilde V)$ is an irreducible component of $V^{n-k}(\omega_X)_0$ of dimension $q$.
Denote by $h\colon Y \to W$ the fibration induced by the abelian variety $Z' \stackrel{{\rm def}}{=} F_{\sE}({\rm id}_X,Z) \subset V^{n-k}(\omega_Y)_0$ as in Lemma \ref{5.1}. 

%
 
\begin{prop}\label{propbasedim}
If  $\dim W<\dim V$, then $Y$ admits a fibration $p'\colon Y\to U$ in $F_Y^{k',q'}$ for some integers $0<k'< \dim V$ and $2\leq q'\leq q(V)$. Moreover,  there exists a dominant rational map $\gamma \colon W \dashrightarrow U$ such that $p'=\gamma \circ h$.  
\end{prop}

\begin{proof}
The proposition follows by Proposition \ref{propdimless} as $q(W)= \dim Z' =q(V) \geq k$ and $\dim W < \dim V = k$.

\end{proof}

 We can say something more about $W$ when $n=4$ and $k=3$. 
We denote by $(\pi'\colon \widetilde Y \to Y,  \; \widetilde h \colon \widetilde Y \to \widetilde W)$ a non-singular representative of $h$.

\begin{prop}\label{cor4}
If $\dim W<3$, then $\widetilde W$ is either  birational to an isotrivial elliptic surface fibered over a curve of genus $q(\widetilde W)-1$, or  a smooth projective curve. 
\end{prop}

\begin{proof}
Suppose   $\dim W=2$. We will prove  that   $\chi(\omega_{\widetilde W})=0$, so that  $\widetilde W$ is    birational to an isotrivial elliptic surface  fibered over a curve of genus $q(\widetilde W)-1$. This fact follows by  the classification theory of complex algebraic surfaces as $\widetilde W$ is of  maximal Albanese dimension and $q(\widetilde W) = q(W) \geq 3$ (\emph{cf}. \cite[Chapter V Section 12]{BHPVdV} and \cite[Theorem X.4]{B}).
If, by contradiction, $\chi(\omega_{\widetilde W})>0$, then by Corollary \ref{corfibration} we   have   $Z' = \pi'_* \widetilde h^*\Pic^0(\widetilde{W})$ is an irreducible component of $V^2(\omega_Y)_0$.  
 In turn, by Theorem \ref{lopo},  the component   $Z = F_{\sE}^{-1} ({\rm id}_Y,Z')$ would be an irreducible component of $V^2(\omega_X)_0$ too, which is impossible as  $Z$ induces a fibration onto the threefold $V$.   The  case $\dim W=1$ is obvious.
\end{proof}

\begin{theorem}\label{fourfolds}
Let $X$ and $Y$ be smooth projective fourfolds, and let  $\Phi_{\sE} \colon \rd(X) \to \rd(Y)$  be an equivalence of triangulated categories. If  $F_X^g=\emptyset$ for all $g \geq 2$, then for any integer $q\geq 3$ the equivalence $\Phi_{\sE}$ induces a bijection of sets $\eta_q\colon F_X^{3,q} \to F_Y^{3,q}$ preserving the bases of the fibrations.
\end{theorem}

\begin{proof}
We define $\eta_q  =  u_{Y,3,q}^{-1} \, \circ \, F_{\sE} \, \circ \, u_{X,3,q}$.
Let $(f\colon X \to V) \in F_X^{3,q}$ and suppose   $\eta_q(f)=(h\colon Y \to W)$. To begin with,  we will prove   $\dim W=3$.
If $X$ does not carry any irrational pencil of genus $g\geq 2$, then so does $Y$ by Theorem \ref{mainthm}. This forces $\dim W>1$. Suppose now $\dim W=2$. By Proposition \ref{cor4}, the variety $\widetilde W$ admits an elliptic fibration onto a smooth curve $B$ of genus $g\geq 2$. By Proposition \ref{extension}, there exists a rational map $Y\dashrightarrow B$ which extends to a morphism. Hence $Y$   admits an irrational pencil of genus $g\geq 2$ which is impossible.
We conclude   $\dim W=3$ and    $\eta_q$  defines a bijection. The proof that $\eta_q$ preserves the bases of the fibrations up to isomorphism  follows by Proposition \ref{deriso}.
\end{proof}

\section{Higher irrational pencils and $K$-equivalence}\label{secK}
In this section we   study the behavior of the sets of fibrations $F_X^{k,q}$ under $K$- and birational equivalence.

Let $X$ and $Y$ be smooth projective varieties.
To begin with, we define a $k$-\emph{higher irrational pencil} $f\colon X\to V$ as a $\chi$-positive $k$-higher irrational pencil without the condition on the positivity of $\chi(\omega_{\widetilde V})$. 
 For every pair of integers $q\geq k\geq 1$, we denote by $P_X^{k,q}$ the set of isomorphism classes of  $k$-higher irrational pencils $f\colon X\to V$  such that $\dim V=k$ and $q(V)=q$. Note   $P_X^{k,q}\supset F_X^{k,q}$. 
The following proposition holds  if $X$ and $Y$ are   $K$-equivalent.

\begin{prop}\label{propp}
 If $X$ and $Y$ are birational, then for every  pair of integers $q\geq k\geq 1$ 
there exists a bijection of sets 
$\tau_{k,q} \colon  P_X^{k,q} \to P_Y^{k,q} $ preserving the bases of the fibrations up to isomorphism.
  \end{prop}

\begin{proof}
Let $\varphi \colon Y \dashrightarrow X$ be a birational map and let $f\colon X\to V$ be a $k$-higher irrational pencil in $P_X^{k,q}$. Moreover consider the composition $h= \big(f \circ \varphi \big)  \colon Y \dashrightarrow V$. 
By Proposition \ref{extension}  $h$ extends to a morphism. The assignment $f\mapsto h$ defines the wanted bijection $\tau_{k,q}$.
\end{proof}
%
%
%

 \subsection*{Acknowledgements}
 I am especially grateful to Christian Schnell for sharing his insight that led to  the proof of Theorem \ref{mainthm}. 
 I am   grateful to an anonymous referee for carefully reading the paper and suggesting several improvements.
    I thank Federico Caucci, Fran\c{c}ois Greer, Georg Hein, Daniel Huybrechts, Rob Lazarsfeld, Anatoly Libgober, Emanuele  Macr\`{i},  Giuseppe Pareschi,
    Mihnea Popa, Elisa Postinghel, Botong Wang and Letao Zhang for fruitful conversations. I acknowledge the support of the Simons Foundation,  grant $261756$ of the Research Councils of Norway,  
  SIR 2014 AnHyC ``Analytic aspects in complex and hypercomplex geometry" (code RBSI14DYEB), and the Rita Levi-Montalcini program.

\bibliographystyle{amsalpha}
\bibliography{bibl}

\end{document}